\documentclass[a4paper]{article}

\usepackage[spanish, mexico, english]{babel}
\usepackage[utf8]{inputenc}
\usepackage{lmodern} % optimize some fonts
\usepackage{titling}
\usepackage{comment}
\PassOptionsToPackage{hyphens}{url}\usepackage{hyperref}
\usepackage{csquotes}
\usepackage{enumerate}
\usepackage[shortlabels]{enumitem}

\usepackage{amsthm, amssymb, amsmath, amsfonts, bbm, mathtools}
\usepackage{float}
\usepackage{mathrsfs}

\usepackage[capitalise]{cleveref}
\usepackage{caption} %para usar caption con *
\usepackage{subcaption}
\usepackage{graphicx}
\usepackage{tikz}
\usetikzlibrary{graphs, graphs.standard}

\usepackage[margin = 1in]{geometry}
\usepackage{layout} %% use \layout
\newcommand{\N}{\mathbb{N}}

\newcommand{\R}{\mathbb{R}}

\newcommand{\abs}[1]{\left\lvert#1\right\rvert}

\newcommand{\Pp}[1]{\mathbb{P}\left(#1\right)}
\newcommand{\E}[1]{\mathbb{E}\left(#1\right)}
\DeclareMathOperator{\var}{Var}
\newcommand{\Var}[1]{\var\left(#1\right)}

\newcommand{\energy}[1]{\mathcal{E}\left(#1\right)}
\tikzstyle{my_style} = [circle, fill = black, inner sep = 0pt, minimum size = 7pt]

\newcommand{\stargraph}[1]{
    \begin{tikzpicture}
        \node[style = {my_style}] at (360:0) (center) {};
        
        \foreach \n in {1, ..., #1}{
            \node[style = {my_style}] at ({\n*360/#1 + 90}:1) (n\n) {};
            \draw (center)--(n\n);
        }
        
    \end{tikzpicture}
}
\newcommand{\pathgraph}[1]{%
    \begin{tikzpicture}
        \graph [branch right] {
            subgraph P_n [n = #1, empty nodes, nodes = {circle, fill = black, inner sep = 0pt, minimum size = 7pt}]
        };
    \end{tikzpicture}
}

\newcommand{\cyclegraph}[1]{%
    \begin{tikzpicture}
        \graph [clockwise] {
            subgraph C_n [n = #1, empty nodes, nodes = {circle, fill = black, inner sep = 0pt, minimum size = 7pt}]
        };
    \end{tikzpicture}
}

\theoremstyle{plain}

\newtheorem{definition}{\protect\defname}
\newtheorem{proposition}{\protect\propname}
\newtheorem{theorem}{\protect\thmaname}
\newtheorem{lemma}{\protect\lemmaname}
\newtheorem{corollary}{\protect\corname}
\newtheorem{example}{\protect\exname}
\newtheorem{remark}{\protect\rename}

\newcommand{\defname}{}\newcommand{\propname}{}\newcommand{\thmaname}{}\newcommand{\lemmaname}{}\newcommand{\corname}{}\newcommand{\exname}{}\newcommand{\rename}{} % initialization
\addto\captionsenglish{%
    \renewcommand{\defname}{Definition}
    \renewcommand{\propname}{Proposition}
    \renewcommand{\thmaname}{Theorem}%
    \renewcommand{\lemmaname}{Lemma}
    \renewcommand{\corname}{Corollary}
    \renewcommand{\exname}{Example}
    \renewcommand{\rename}{Remark}%
}
\addto\captionsspanish{%
    \renewcommand{\defname}{Definici\'on}
    \renewcommand{\propname}{Proposici\'on}
    \renewcommand{\thmaname}{Teorema}%
    \renewcommand{\lemmaname}{Lema}
    \renewcommand{\corname}{Corolario}
    \renewcommand{\exname}{Ejemplo}\renewcommand{\rename}{Observaci\'on}%
}

\title{Bound for the energy of graphs in terms of degrees and leaves}
\author{Octavio Arizmendi, Samuel Gurrola-Viramontes}
\date{\today}

\begin{document}

\selectlanguage{english}
\maketitle

\abstract{ We provide a new upper bound for the energy of graphs in terms of degrees and number of leaves. We apply this formula to study the energy of Erd\"os-R\'enyi graphs and Barabasi-Albert trees.}

~

\section{Introduction}

In this paper, we focus on the concept of graph energy which is a key metric in spectral graph theory. The energy of a graph is defined as the sum of the absolute values of the eigenvalues $\lambda_1, \ldots, \lambda_n$ of its adjacency matrix, that is
\[
\mathcal{E}(G)= \sum_{i=1}^{n} |\lambda_i|.
\]

Graph energy has been a subject of increasing interest due to its connections to various fields, including chemistry, where it is used to model molecular stability, and network theory, where it provides insights into structural properties. An excellent and comprehensive introduction to the theory of graph energy can be found in the monograph \cite{Li12}, which offers a broad survey of results and methods related to this topic.

In this paper, we contribute to the study of graph energy by establishing a novel inequality that relates the energy $\mathcal{E}(G)$ of a graph to the degrees of its vertices and the number of leaves adjacent to each vertex. Our result provides a deeper understanding of how local structural features, such as vertex degree and adjacency to leaves, influence the global spectral properties of the graph.

To be more precise, we prove the following inequality holds for any graph \(G = (V,E)\),
\[
    \mathcal{E}(G) \leq 2e_{11} + \sum_{v \in V'} \sqrt{3l(v) + d(v)},
\]
where $e_{11}$ is the numbers of edges joining two leaves, $l(v)$ denotes the number of leaves connected to the vertex $v$ and \(V'\) denotes the set of inner vertices of \(G\).

This inequality improves some of the previous upper bounds as we explain in Section \ref{result_section} and it is specially suitable for trees, since in many cases, the number of leaves is large.  

Despite the various papers proving inequalities for the energy and properties for general graphs,  research exploring graph energy in the context of random graphs remains relatively sparse. In particular, random trees present specific challenges as distribution of eigenvalues is not well understood, and, consequently, the asymptotic behavior of the energy is hard to tackle. One of the motivations to derive the above inequality is to provide a closer bound for Barab\'asi-Albert Tree and  sparse Erd\"os-Renyi graphs.

In Section \ref{BA_section}, we apply this inequality to improve the upper bound for the energy of a Barab\'asi-Albert tree and an Erd\"os-Renyi graph.

\section{Preliminaries}

\subsection{Notation and definitions on graphs}

A \textit{graph} \(G\) is an ordered pair of sets \((V(G),E(G))\) such that the elements of \(E(G)\) are unordered pairs of different elements in \(V(G)\). The elements of \(V(G)\) are called the \textit{vertices} of the graph \(G\) and the elements of \(E(G)\) are called its \textit{edges}. To avoid ambiguities in our notation, we will always assume that \(V(G) \cap E(G) = \emptyset\). If there is no risk of confusion, we will write \(V := V(G)\) and \(E := E(G)\).

Let \(v,w\) vertices of a graph \(G\). If \(e = \{v,w\}\) is an edge of \(G\), we say that \(v,w\) are the \textit{endpoints} of the edge \(e\). We also say that \(v,w\) are \textit{neighbors}.% or that they are \textit{connected}
We denote it by \(v \sim w\). We define
\[
    N(v) := \{w \in V : v \sim w\}
\]
as the \textit{neighborhood} of \(v\). In addition, we define the \textit{degree} of \(v\) as \(d(v) := \abs{N(v)}\), i.e. the degree of \(v\) is the number of vertices that are connected to \(v\). In particular, if \(d(v) = 1\), we say that \(v\) is a \textit{leaf}.

In contrast, if a node is not a leaf nor an isolated node (with degree zero), we say that it is an \textit{inner vertex} or an \textit{inner node}. We will denote the number of inner vertices connected to a vertex \(v\) by
\[
    \delta(v) := \abs{\{w \in V : v \sim w, \, d(w) > 1\}}.
\]
For a non-empty graph \(G\) with vertices $v$ and $w$, a path from $v$ to $w$ is a sequence of vertices $v_1,\dots,v_l$, with $v_i\sim v_{i+1}$, such that  $v_1=v,v_l=w$. A graph is said to be \textit{connected} if any two vertices are linked by a path in \(G\).  If a graph is not connected, we say that it is \textit{disconnected}. We can talk about the \textit{connected components} of the graph \(G\), which naturally are the maximal connected sub-graphs of \(G\). Some examples of connected graphs will be given in the next section.

Let \(G\) be a graph with exactly \(n\) vertices. We can list the vertices of \(G\) as \(v_1, \ldots, v_n\). The \textit{adjacency matrix} \(A_G\) of \(G\), with entries \((a_{ij})\), is a matrix of size \(n \times n\) such that
\[
    a_{ij} = \begin{cases}
        1 & \mbox{if } v_i \sim v_j,\\
        0 & \mbox{otherwise.}
    \end{cases}
\]

The \textit{characteristic polynomial} of a graph \(G\) is exactly the characteristic polynomial of its adjacency matrix \(A_G\). It will be denoted by \(\phi_G\). The \textit{eigenvalues} of \(G\) are defined as the eigenvalues of \(A_G\). Notice that the eigenvalues of a graph are real numbers since the adjacency matrix of a graph is a real symmetric matrix.

In addition, we can consider a \textit{weighted graph} in the following way. Let \(G = (V, E)\) be a graph with set of vertices \(V = \{v_1, \ldots, v_n\}\). Then, we assign a \textit{weight} \(w_{ij} \in \R\) to every edge \(\{v_i,v_j\} \in E\). Since the edge \(\{v_i,v_j\}\) is the same as the edge \(\{v_j,v_i\}\), we set \(w_{ij} = w_{ji}\). If two vertices \(v_i,v_j \in V\) are not connected, we establish \(w_{ij} = 0\). In this case, we define the \textit{weight matrix} $W_G$ associated to $G$ as the matrix of size \(n \times n\) with entries \((w_{ij})_{i,j}\) and we denote the weighted graph by $(G,W_G)$. If there is no risk of confusion we will write $W := W_G$.

The characteristic polynomial and the eigenvalues of a weighted graph are defined as the characteristic polynomial and the eigenvalues of its weight matrix, respectively. The eigenvalues of a weighted graph are real numbers again. Note that a graph is a weighted graph where each of its edges has associated a weight of value one.

%\textcolor{red}{Let \(G = (V,E)\) and \(H = (W,F)\) two graphs. If \(V \subseteq W\) and \(E \subseteq F\) we say that \(G\) is a \textit{sub-graph} of \(H\). It is usually denoted by \(G \subseteq H\).}

%\textcolor{red}{If \(G \subseteq H\) and \(G\) contains all the edges \(\{u,v\} \in F\) with \(u,v \in V\), then \(G\) is an \textit{induced sub-graph} of \(H\). We say that \(V\) \textit{induces} \(G\) in \(H\) and we write \(H[V] := G\).}

\subsection{Some families of graphs}
Some of the most common types of graphs are paths, stars, cycles and complete graphs. A \textit{path} is a non-empty graph \(P = (V,E)\) of the form
\[
    V = \{v_1, \ldots, v_n\}, \qquad E = \{\{v_1,v_2\}, \{v_2, v_3\}, \ldots, \{v_{n-1},v_n\}\},
\]
where \(v_1, \ldots, v_n\) are all distinct. We denote this graph by \(P_n\), where \(n\) is the number of vertices of this graph.

\begin{figure}[H]
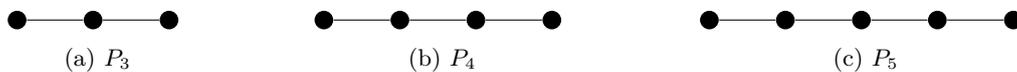

    \centering
    \begin{subfigure}[b]{0.25\textwidth}
        \centering
        \pathgraph{3}
        \caption{\(P_3\)}
    \end{subfigure}%
    \begin{subfigure}[b]{0.32\textwidth}
        \centering
        \pathgraph{4}
        \caption{\(P_4\)}
    \end{subfigure}%
    \begin{subfigure}[b]{0.38\textwidth}
        \centering
        \pathgraph{5}
        \caption{\(P_5\)}
    \end{subfigure}%
    \caption{Some examples of path graphs.}
    %\label{fig:enter-label}
\end{figure}

A \textit{star} is a non-empty graph \(S = (V,E)\) of the form
\[
    V = \{v_0, v_1, \ldots, v_n\}, \qquad E = \{\{v_0, v_1\}, \{v_0, v_2\}, \ldots, \{v_0, v_n\}\},
\]
where \(v_0, v_1, \ldots, v_n\) are all distinct. We denote this graph by \(S_n\), where \(n\) is the number of edges of this graph.

\begin{figure}[H]
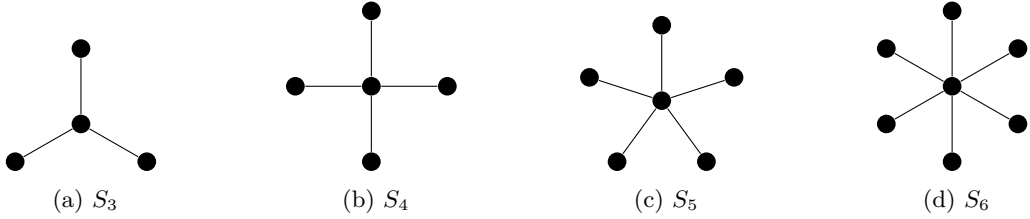

    \centering
    \begin{subfigure}[b]{0.24\textwidth}
        \centering
        \stargraph{3}
        \caption{\(S_3\)}
    \end{subfigure}%
    \begin{subfigure}[b]{0.24\textwidth}
        \centering
        \stargraph{4}
        \caption{\(S_4\)}
    \end{subfigure}%
    \begin{subfigure}[b]{0.24\textwidth}
        \centering
        \stargraph{5}
        \caption{\(S_5\)}
    \end{subfigure}%
    \begin{subfigure}[b]{0.24\textwidth}
        \centering
        \stargraph{6}
        \caption{\(S_6\)}
    \end{subfigure}
    \caption{Some examples of a star graph.}
\end{figure}

On the other hand, a \textit{cycle} is a non-empty graph \(C = (V,E)\) of the form
\[
    V = \{v_1, \ldots, v_n\}, \qquad E = \{\{v_1, v_2\}, \{v_2, v_3\}, \ldots, \{v_{n-1}, v_n\}, \{v_n, v_1\}\},
\]
where \(v_1, \ldots, v_n\) are all distinct. We denote this graph by \(C_n\), where \(n\) is the number of vertices (and edges) of this graph.

\begin{figure}[H]
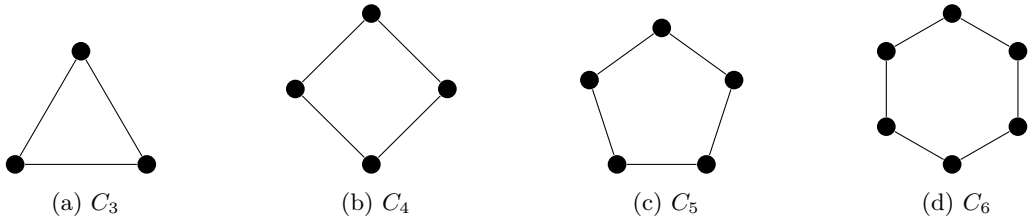

    \centering
    \begin{subfigure}[b]{0.24\textwidth}
        \centering
        \cyclegraph{3}
        \caption{\(C_3\)}
    \end{subfigure}%
    \begin{subfigure}[b]{0.24\textwidth}
        \centering
        \cyclegraph{4}
        \caption{\(C_4\)}
    \end{subfigure}%
    \begin{subfigure}[b]{0.24\textwidth}
        \centering
        \cyclegraph{5}
        \caption{\(C_5\)}
    \end{subfigure}%
    \begin{subfigure}[b]{0.24\textwidth}
        \centering
        \cyclegraph{6}
        \caption{\(C_6\)}
    \end{subfigure}
    \caption{Some examples of a cycle graphs.}
\end{figure}

Finally, a \textit{complete graph} is a non-empty graph with all the possible edges. We denote this graph by \(K_n\), where \(n\) is the number of vertices of this graph.

\subsection{Sachs' theorem for weighted graphs}

We now introduce Sach's theorem and its generalization to weighted graphs. This gives a way to write the characteristic polynomials in terms of specific graph counting. In order to do this we need now to define a special type of graphs.

\begin{definition}[Sachs' graph]
    A graph \(S\) is called a Sachs' graph if its connected components are isomorphic to \(K_2\) or to a cycle. Associated to a Sachs' graph \(S\), we can define \(r(S)\) to be the number of connected components of \(S\) and \(c(S)\) to be the number of connected components in \(S\) isomorphic to a cycle.
    
\end{definition}

As announced above, there is a close connection between this type of sub-graphs and the characteristic polynomial of a graph given by the next theorem.
\begin{theorem}[Sachs' theorem]\label{ST}
    Let \(G\) be a graph in \(n\) vertices and \(\phi_G(x) = \sum_{k = 0}^n b_k x^{n-k}\) be its characteristic polynomial. The coefficients \((b_k)_{k \geq 1}\) satisfies the equality
    \[
        b_k = \sum_{S \in \mathcal{S}_k(G)} (-1)^{r(S)}2^{c(S)},
    \]
    where \(\mathcal{S}_k(G)\) denotes the set of sub-graphs of \(G\) with exactly \(k\) vertices that are Sachs' graphs. Furthermore, \(b_0 = 1\).
    
\end{theorem}

The proof of this theorem can be reviewed in  \cite{Cvectovik}.
Following \cite{Mariya05}, we will consider a simple extension for the Sachs' theorem.

\begin{definition}[Sachs' sub-graph for a weighted graph]
    Let \(G\) be a weighted graph in \(n\) vertices \(v_1, \ldots, v_n\) with weight matrix \(W_G = (w_{ij})_{i,j}\). A sub-graph \(S\) of \(G\) is called a Sachs' sub-graph if its connected components are isomorphic to a weighted \(K_2\) or to a weighted cycle. Associated to \(S\), we can define the function \(W(S)\) which is the product of the weights of its connected components. The weight of a component isomorphic to a weighted \(K_2\) is \(w_{i,j}^2\) the square of the weight of its only edge \(\{v_i,v_j\}\) and the weight of a component isomorphic to a weighted cycle is the product of the weights \(w_{i,j}\) of all its edges \(\{v_i,v_j\}\). In addition, we define \(r(S)\) and \(c(S)\) to be the number of connected components of \(S\) and the number of connected components in \(S\) isomorphic to a cycle, respectively.
    
\end{definition}

\begin{theorem}[Sach's theorem for weighted graphs]
    Let \(G\) be a weighted graph in \(n\) vertices and \(\phi_G(x) = \sum_{k = 0}^n b_k x^{n-k}\) its characteristic polynomial. The coefficients \((b_k)_{k \geq 1}\) satisfies the equality
    \[
        b_k = \sum_{S \in \mathcal{S}_k(G)} (-1)^{r(S)}2^{c(S)}W(S),
    \]
    where \(\mathcal{S}_k(G)\) denotes the set of Sachs' sub-graphs of \(G\) with exactly \(k\) vertices. Furthermore, \(b_0 = 1\).
    
\end{theorem}

The proof of the above result follows the same lines as the proof of the usual Sachs' theorem.

\subsection{Inequalities for the Energy of a Graph}

We define the \textit{energy} of a graph (weighted graph) \(G\) by
\[
    \energy{G} := \sum_{k = 1}^n \abs{\lambda_k},
\]
where \(\lambda_1, \ldots, \lambda_n\) are the eigenvalues associated to the graph \(G\).

We will review some useful inequalities and properties for the energy of a graph. The first inequality is McClelland's inequality \cite{McClelland71} which states that for a graph \(G\) with \(n\) vertices and \(m\) edges, it satisfies that
\begin{equation}\label{McClelland}
    \energy{G} \leq \sqrt{2mn},
\end{equation}
with equality if and only if \(G\) is isomorphic to \(\frac{n}{2}\) copies of \(K_2\).

Another inequality due to Koolen and Moulton \cite{Koolen01} is that for any graph \(G\) with \(n\) vertices and \(m\) edges such that \(2m \geq n\), we have
\begin{equation}\label{Koolen}
    \energy G \leq \frac{2m}{n} + \sqrt{(n-1)\left(2m - \left(\frac{2m}{n}\right)^2\right)}.
\end{equation}
This equality holds if and only if \(G\) is isomorphic to \(\frac{n}{2}\) copies of \(K_2\), is isomorphic to \(K_n\), or is a non-complete connected strongly regular graph with two non-trivial eigenvalues both having absolute values equal to \(\sqrt{(2m - (2m/n)^2)/(n-1)}\).

More recently, two bounds for the energy of the a graph in terms of the degrees of the graph were obtained.

\begin{proposition}[Arizmendi and Juarez \cite{Arizmendi18}]\label{eq:AJ}
For a graph \(G\) with vertices of degrees \(d_1, \ldots, d_n\) it is satisfied
\begin{equation}\label{eq:AJ2}
    \energy G \leq \sum_{i = 1}^n \sqrt{d_i}.
\end{equation}

\end{proposition}

\begin{proposition}[Arizmendi and Dominguez \cite{Arizmendi22}]\label{eq:AD}
For a tree \(T\) with vertices degrees \(\Delta = d_1 \geq \ldots \geq d_n\), \(n \geq 3\), it is satisfied
\begin{equation}\label{eq:AD2}
    \energy T \leq \sum_{i = 2}^n 2\sqrt{d_1 - 1} + 2\sqrt{\Delta} \leq \sum_{i = 1}^n 2\sqrt{d_i - 1} + 1.
\end{equation}

\end{proposition}

Finally, the following theorem of Ky-Fan is very useful. It shows that the energy of graphs satisfies a sort of triangle inequality. 
\begin{theorem}[Ky-Fan's theorem for weighted sub-graphs]
    Let \(G\) be a weighted graph and \(H_1, \ldots, H_n\) be weighted sub-graphs of \(G\) whose weight matrices satisfy that \(W_G = W_{H_1} + \ldots + W_{H_n}\). Then
    \[
        \energy{G} \leq \energy{H_1} + \ldots + \energy{H_n},
    \]
    with equality if and only if \(n = 1\).
\end{theorem}

As pointed out in \cite{Arizmendi22}, the above inequalities \eqref{eq:AJ2} and \eqref{eq:AD2} can be obtained from Ky- Fan's theorem by considering suitable sub-graphs. Our aim is to improve the above bound by considering instead weighted sub-graphs.

\section{Main Result} \label{result_section}
In this section we prove the main inequality as described in the introduction. Then we compare with the known inequalities mentioned in the the preliminaries.
\subsection{An inequality for the energy in terms of degrees and leaves}

Our main result comes from the idea of thinking a graph as the union of several weighted stars. This decomposition, together with Ky-Fan's theorem, will help us to find a new inequality for the energy of a graph.  

Before going into the details of the main proof let us state a useful lemma.

\begin{lemma}\label{lema:weighedstar}
Let $S:=(S_n,W)$ be a weighted star with $n$ edges and weights $(w_{i})_{i = 1}^n$. Then,
$$\mathcal{E}(S) = 2\sqrt{\sum_{i = 1}^n w_i^2}.$$

\end{lemma}

\begin{proof}
By Sachs' theorem for weighted graphs, we get the characteristic polynomial $\phi_S$ of \(S\) given by
\begin{align*}
    \phi_{S}(x) &= x^{n + 1} - \left(\sum_{i = 1}^n w_i^2\right)x^{n - 1} = x^{n - 1}\left(x^2 - \sum_{i = 1}^n w_i^2\right).
\end{align*}
    
We conclude that the eigenvalues of \(S\) are 0, with multiplicity $n-1$ and \(\pm\sqrt{\sum_{i = 1}^n w_i^2}\). Therefore,
\[
    \energy{S_n} = 2\sqrt{\sum_{i = 1}^n w_i^2}.
\]
    
\end{proof}
Now, we need to establish a couple of definitions, in order to state our main result.
Let \(G = (V, E)\) be a graph and \(v \in V\) a vertex of \(G\). We define the set of neighbors of \(v\) such that they are leafs as \(L(v)\), i.e.
\[
    L(v) := \{w \in V : v \sim w,\, d(w) = 1\},
\] 
and \(l(v) := \abs{L(v)}\) as the number of neighbors of \(v\) which are leafs. Note that \(L(v) \subseteq N(v)\). Let \(L(G)\) be the set of all leafs of the graph \(G\). If there is no risk of confusion, we will write $L := L(G)$. Besides, let \(V'(G)\) be the set of inner vertices of \(G\), or the set of all vertices in \(G\) which are not leafs, i.e.
\[
    V'(G) := V \setminus L = \{v \in V : d(v) \geq 2\}.
\]
If there is no risk of confusion, we will write \(V' := V'(G)\).

The number of edges of a graph $G$ with both endpoints being leaves will be helpful later on. Let denote it by $e_{11}(G)$. If there is no risk of confusion, we will write $e_{11} := e_{11}(G)$.

\begin{theorem}\label{TP}
    Every connected graph \(G = (V,E)\) with at least three nodes satisfies the inequality
    \[
        \energy{G} \leq \sum_{v \in V'} \sqrt{3l(v) + d(v)} = \sum_{v \in V'} \sqrt{4l(v) + \delta(v)}.
    \]
    This equality holds if and only if \(G\) is isomorphic to a star.
    
\end{theorem}

\begin{proof}
    Let \(v \in V'\) and define \(S(v)\) as the weighted sub-graph of \(G\) induced by \(N(v) \cup \{v\}\). We assign weights to every edge \(\{v,w\}\) of \(S(v)\) as follows:
    \begin{itemize}
        \item if \(w \in L(v)\), we assign a weight of 1 to \(\{v,w\}\),
        \item if \(w \notin L(v)\), we assign a weight of \(\frac{1}{2}\) to \(\{v,w\}\).
    \end{itemize}

    Notice that
    \[
        G = \bigcup_{v \in V'} S(v).
    \]
    Indeed, it is clear that
    \[
        \bigcup_{v \in V'} V(S(v)) \subseteq V(G).
    \]
    On the other hand, if \(v \in L\), there is a unique \(w \in V\) such that \(v \sim w\). Since \(G\) is connected and it has at least three vertices, we have \(v \in N(w)\) with \(w \in V'\). Therefore, since \(V(S(v)) = N(v) \cup \{v\}\), one can see that
    \[
        V(G) \subseteq \bigcup_{v \in V'} V(S(v))
    \]
    and we can conclude that
    \[
        V(G) = \bigcup_{v \in V'} V(S(v)).
    \]
    In a similar way, we can see that
    \[
        E(G) = \bigcup_{v \in V'} E(S(v)).
    \]
    
    It only remains to check that each edge has a total weight of 1 when considering the decomposition into weighted stars of the graph \(G\). Let \(\{v,w\} \in E\) be an edge of \(G\). Since \(V = L \cup V'\) with \(L \cap V' = \emptyset\), we examine two cases:
    \begin{itemize}
        \item If \(v \in L\), then the edge \(\{v,w\}\) is the only edge that contains to \(v\) as an endpoint. The vertex \(v\) appears only once in the star decomposition of \(G\) and it is in \(S(w)\) with weight 1. It is the same analysis when \(w \in L\).
        \item If \(v,w \in V'\), the edge \(\{v,w\}\) appears only twice in the star decomposition of \(G\). It appears once in \(S(v)\) with weight \(\frac{1}{2}\) and another time in \(S(w)\) with the same weight of \(\frac{1}{2}\). At the end, the edge \(\{v,w\}\) has a total weight of 1.
    \end{itemize}

    By using Ky-Fan theorem for weighted sub-graphs, we get the inequality
    \[
        \energy G \leq \sum_{v \in V'} \energy{S(v)},
    \]
    where the equality holds if and only if \(G\) has just one weighted star component, i.e., if and only if \(G\) is isomorphic to a star.

Consequently, by Lemma \ref{lema:weighedstar}
    \[
        \energy G \leq \sum_{v \in V'} \sqrt{3l(v) + d(v)} = \sum_{v \in V'} \sqrt{4l(v) + \delta(v)},
    \]
    where the last inequality holds because \(d(v) = l(v) + \delta(v)\) for every vertex \(v \in V\).
    
\end{proof}

We can generalize the result from Theorem \ref{TP} to a general graph \(G\) as follows.

\begin{theorem}\label{TPG}
    For a graph \(G = (V,E)\) we have
    \[
    \energy G \leq 2 e_{11} + \sum_{v \in V'} \sqrt{3l(v) + d(v)}.
    \]
    
\end{theorem}

\begin{proof}
    Let \(G = (V,E)\) be a graph with connected components \(H_1, \ldots, H_n\). Since the energy is linear over connected components, we obtain
    \[
    \energy G = \sum_{k = 1}^n \energy{H_k}.
    \]

    Notice there are three types of connected components: isolated vertices which have energy zero; isolated edges which have energy 2; connected graphs with at least three vertices. Since
    \[
        V'(G) = \bigcup_{k = 1}^n V'(H_k)
    \]
    where \(V'(H_1), \ldots V'(H_n)\) are disjoint by pairs and \(V'(H) = \emptyset\) for a graph \(H\) isomorphic to an edge, we conclude by Theorem \ref{TP} that
    \[
        \energy G \leq 2 e_{11} + \sum_{v \in V'} \sqrt{3l(v) + d(v)}.
    \]
    
\end{proof}

We can also have a global version of our main theorem.
 
\begin{corollary}\label{global_bound}
    For a graph \(G = (V,E)\) without isolated nodes
    it is satisfied that
    \[
        \energy G \leq  2 e_{11} + \sqrt{2(\abs{V} - \abs{L})(\abs{E} + \abs{L} - 3 e_{11})}.
    \]
    
\end{corollary}

\begin{proof}
    Let \(G = (V,E)\) be a graph. By the theorem below and Cauchy-Schwartz inequality, we get the next series of inequalities
    \begin{align*}
        \energy G &\leq \sum_{v \in V'} \sqrt{3l(v) + d(v)}\\
        &\leq \sqrt{\abs{V'}\left(\sum_{v \in V'} 3l(v) + d(v)\right)}.
    \end{align*}

    Since $G$ does not have isolated vertices, we have that \(\abs{V'} = \abs{V} - \abs{L}\). Note also that for every leaf \(v \in L\) of \(G\) there exist a unique \(w \in L\) where \(v,w\) is and edge counted in \(e_{11}\) or a unique \(w \in V'\) such that \(v \sim w\). Then, for every $v \in L$, it is satisfied that \(l(v) = 1\) if \(v\) is endpoint of an edge counted in \(e_{11}\) or \(l(v) = 0\) otherwise. It follows that
    \[
        \sum_{v \in V'} l(v) =  \abs{L} - 2e_{11}.
    \]

    Using that \(\sum_{v \in V} d(v) = 2\abs{E}\), we obtain in a similar way that
    \[
        \sum_{v \in V'} d(v) = 2\abs{E} - \abs{L}.
    \]

    Placing all together, we obtain
    \begin{align*}
        \energy G &\leq 2e_{11} + \sqrt{2(\abs{V} - \abs{L})(\abs{E} + \abs{L} - 3e_{11})}.
    \end{align*}
    
\end{proof}

\begin{remark}
    Note that for a graph \(G = (V,E)\) with \(e_{11} = 0\) the above inequality is equivalent to
    \[
        \energy G \leq \sqrt{2(\abs{V} - \abs{L})(\abs{E} + \abs{L})}.
    \]
    In particular, for a connected graph with at least three vertices this inequality holds. In addition, if \(\abs{V} \leq \abs{E} + \abs{L}\) holds, this bound is smaller than McClelland's bound \eqref{McClelland}.
    
\end{remark}

\begin{remark}
    Note that for a graph \(G = (V,E)\) with set of isolated vertices \(I\) we can infer the following inequality,
    \[
        \energy G \leq 2 e_{11} + \sqrt{2(\abs{V} - \abs{L} - \abs{I})(\abs{E} + \abs{L} - 3 e_{11})}.
    \]
    
\end{remark}

A question we can ask is when the local bound and the global bound for the energy of a graph coincide. The next example shows an answer for a particular case.

\begin{example}
For a connected graph with at least three vertices, according to Cauchy-Schwartz inequality, the equality between local and global bounds are equal if and only if
\begin{equation}\label{pinatas}
    4l(v) + \delta(v) = k, \qquad \mbox{for every } v \in V',
\end{equation}
for some constant \(k \in \N\). Equivalently,
\[
    \delta(v) = k - 4l(v), \qquad \mbox{for every } v \in V'.
\]

It means if the equality between local and global bounds are the same, then the inner degree of inner vertices differs by a multiple of 4. In the particular case when \(d(v) = d\) and \(l(v) = l\)  for some constants \(d,l \in \N\) and every \(v \in V'\), we obtain this graph as a copy of a \(d\)-regular graph and then attaching \(l\) leaves to every vertex of the \(d\)-regular graph.

\begin{figure}[!htp]
    \centering
    \begin{subfigure}[b]{0.3\textwidth}
        \centering
        \begin{tikzpicture}
            \foreach \n in {1, ..., 7}{
                \node[style = {my_style}] at ({\n*360/7 + 90}:1) (\n) {};
            }

            \draw (1) -- (2) -- (3) -- (4) -- (5) -- (6) -- (7) -- (1);

            \foreach \n in {1, ..., 14}{
                \node[style = {my_style}] at ({(\n + 0.5)*360/14 + 90}:1.8) (n\n) {};) 
            }

            \draw (n1) -- (1) -- (n2);
            \draw (n3) -- (2) --  (n4);
            \draw (n5) -- (3) --  (n6);
            \draw (n7) -- (4) --  (n8);
            \draw (n9) -- (5) --  (n10);
            \draw (n11) -- (6) --  (n12);
            \draw (n13) -- (7) -- (n14);
            
        \end{tikzpicture}
        \caption{}
    \end{subfigure}%
    \begin{subfigure}[b]{0.3\textwidth}
        \centering
        \begin{tikzpicture}
            \foreach \n in {1, ..., 7}{
                \node[style = {my_style}] at ({(\n - 1)*360/7 + 90}:1) (\n) {};
            }

            \foreach \n in {2, ..., 7}{
                \draw (1) -- (\n);
            }

            \draw (2) -- (5); \draw (3) -- (6); \draw (4) -- (7);

            \node[style = {my_style}] at (90:1.8) (n11) {};
            \draw (n11) -- (1);

            \foreach \n in {2, ..., 7}{
                \node[style = {my_style}] at ({(2*\n - 2.5)*360/14 + 90}:1.8) (n1\n) {};
                \node[style = {my_style}] at ({(2*\n - 1.5)*360/14 + 90}:1.8) (n2\n) {};
                \draw (n1\n) -- (\n) -- (n2\n);
            }
            
        \end{tikzpicture}
        \caption{}
    \end{subfigure}%
    \begin{subfigure}[b]{0.3\textwidth}
        \centering
        \begin{tikzpicture}
            \foreach \n in {1, ..., 7}{
                \node[style = {my_style}] at ({(\n - 0.5)*360/7 + 90}:1) (\n) {};
            }

            \foreach \n in {2, ..., 7}{
                \draw (1) -- (\n);
            }

            \foreach \n in {2, ..., 6}{
                \draw (7) -- (\n);
            }

            \node[style = {my_style}] at (360/14 + 90:1.8) (n11) {};
            \draw (n11) -- (1);

            \node[style = {my_style}] at (-360/14 + 90:1.8) (n17) {};
            \draw (n17) -- (7);

            \foreach \n in {2, ..., 6}{
                \node[style = {my_style}] at ({(2*\n - 1.5)*360/14 + 90}:1.8) (n1\n) {};
                \node[style = {my_style}] at ({(2*\n - 0.5)*360/14 + 90}:1.8) (n2\n) {};
                \draw (n1\n) -- (\n) -- (n2\n);
            }
            
        \end{tikzpicture}
        \caption{}
    \end{subfigure}
    \caption{Some examples of connected graphs where the equality between local and global bound holds.}
\end{figure}
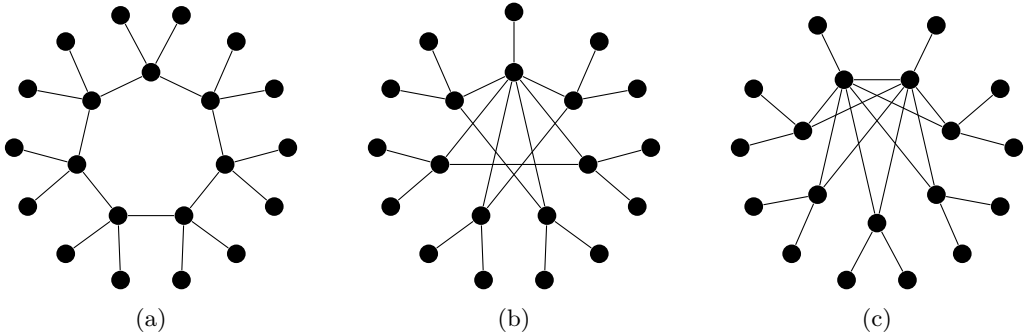

\end{example}

\subsection{Comparison with previous bounds}

To end this section we would like to compare our main inequality with the known upper bounds for the energy of graphs.

First, Theorem \ref{TPG} improves the inequality of Proposition \ref{eq:AJ}. Indeed, let \(G\) be a graph. Since \(\sqrt{3l + d} \leq l + \sqrt{d}\) for every \(l,d \geq 1\), we get
    \[
        \sum_{v \in V'} \sqrt{3l(v) + d(v)} \leq L + \sum_{v \in V'} \sqrt{d(v)} = \sum_{v \in V} \sqrt{d(v)}.
    \]

On the other hand, making a comparison with respect to the inequality of Proposition \ref{eq:AD} is a bit more subtle. A term-to-term comparison tell us that the inequality in Theorem \ref{TP} is better than the bound by Arizmendi and Dominguez, given in \eqref{eq:AD2},  if and only if \(2 \leq \delta(v)\) for every \(v \in V'\). However, this condition is impossible to be satisfied by a tree.

We now show an example and a criterion to decide which inequality is better.

\begin{example}
    The bound offered by Theorem \ref{TP} is smaller than the bound of Proposition \ref{eq:AJ} for a double stars \(S_{p,q}\) (obtained by connecting with an edge the centers of a star \(S_p\) and a star \(S_q\)) when \(p \leq q \leq 3p - 1\).

    In fact, let \(S_{p,q}\) be a double star with \(p \leq q\) for which the bound of Theorem \ref{TP} improves the bound of Proposition \ref{eq:AJ}. We have
    \[
        \sqrt{4p + 1} + \sqrt{4q + 1} \leq 2\left(\sqrt{p} + \sqrt{q + 1}\right),
    \]
    which is equivalent to the inequality
    \[
        \sqrt{\left(p + \frac{1}{4}\right)\left(q + \frac{1}{4}\right)} \leq \frac{1}{4} + \sqrt{p\left(q + 1\right)}
    \]
    Particularly, if \(q \leq 3p - 1\) then the above inequality holds.
    
\end{example}

\begin{theorem}
    Let \(f\) be the real function defined by
    \[
        f(x,y) = \sqrt{4x + 4(y-1)} - \sqrt{4x + y}, \qquad \mbox{for } x, y \geq 1,
    \]
    \(T = (V,E)\) be a tree with \(n \geq 3\) vertices and \(V_1, V_2\) be the subsets of vertices defined by
    \[
        V_1 := \{v \in V' : \delta(v) = 1\}, \qquad V_2 := \{v \in V' : \delta(v) \geq 2\}.
    \]
    Theorem \ref{TP} improves Proposition \ref{eq:AD} for the energy of \(T\) if one of the following conditions hold:
    \begin{enumerate}[i)]
        \item \(\sum_{v \in V'} f(l(v),\delta(v)) \geq 0;\)
        \item \(\sum_{v \in V_1} f(l(v),1) + \sum_{v \in V_2} f(l(v),2) \geq 0;\)
        \item \(\left(\sqrt{l_2 + 1} - \sqrt{l_2 + 1/2}\right)\abs{V_2} \geq \left(\sqrt{l_1 + 1/4} - \sqrt{l_1}\right)\abs{V_1},\) where \(l_1 := \min_{v \in V_1} l(v);\) and \(l_2 := \max_{v \in V_2} l(v);\)
        \item \(2\left(\sqrt{n} - \sqrt{n-1/2}\right)\abs{V_2} \geq \abs{V_1}.\)
    \end{enumerate}
    
\end{theorem}

\begin{proof}
    We have that Theorem \ref{TP} improves Proposition \ref{eq:AD} for the energy of \(T\) if and only if
    \[
        \sum_{v \in V'} \sqrt{3l(v) + d(v)} \leq \sum_{v \in V} 2\sqrt{d(v) - 1} + 2\left(\sqrt{\Delta} - \sqrt{\Delta - 1}\right),
    \]
    where \(\Delta\) is the maximum degree of the vertices in \(T\).
    
    Using that \(\sqrt{d(v) - 1} = 0\) for every \(v \in L\) and \(d(v) = \delta(v) + l(v)\), we obtain that the above inequality is equivalent to
    \begin{equation}\label{AD-TP}
        0 \leq 2\left(\sqrt{\Delta} - \sqrt{\Delta - 1}\right) + \sum_{v \in V'} \sqrt{4l(v) + 4(\delta(v) - 1)} - \sqrt{4l(v) + \delta(v)}.
    \end{equation}
    In particular, if the inequality
    \[
        \sum_{v \in V'} f(l(v),\delta(v)) \geq 0
    \]
    holds, the inequality in \eqref{AD-TP} is satisfied. In consequence, Theorem \ref{TP} improves Proposition \ref{eq:AD} for the energy of \(T\).

    We will show that condition ii) implies condition i), condition iii) implies condition ii) and condition iv) implies condition iii) to complete the proof.
    
    Notice now that \(f(x,y)\) is a decreasing function over \(x\) for \(y \geq 4/3\), it is an increasing function over \(x\) for \(y \leq 4/3\) and it is an increasing function over \(y\) for every \(x \geq 1\). This implies the following inequalities:
    \begin{align*}
        \sum_{v \in V_1} f(l(v),\delta(v)) \geq \sum_{v \in V_1} f(l(v),1), \qquad \sum_{v \in V_2} f(l(v),\delta(v)) \geq \sum_{v \in V_2} f(l(v),2).
    \end{align*}
    It means condition ii) implies condition i).

    In a similar way, we have
    \[
        \left(\sqrt{l_1 + 1/4} - \sqrt{l_1}\right)\abs{V_1} = -f(l_1,1)\abs{V_1} \geq \sum_{v \in V_1} -f(l(v),1)
    \]
    and
    \[
         \sum_{v \in V_2} f(l(v),2) \geq f(l_2, 2)\abs{V_2} = \left(\sqrt{l_2 + 1} - \sqrt{l_2 + 1/2}\right)\abs{V_2},
    \]
    i.e., condition iii) implies condition ii).

    Finally, we have that
    \[
        \sqrt{l_1 + 1} - \sqrt{l_1 + 1/2} \geq \sqrt{n} - \sqrt{n-1/2} \qquad \mbox{and} \qquad 1/2 \geq \sqrt{l_2 + 1/4} - \sqrt{l_2}
    \]
    since \(0 \leq l(v) \leq n-1\). In consequence, condition iv) implies condition iii).
    
\end{proof}

Having established general properties of our bound we proceed to apply them to random graphs.

\section{Energy for Barab\'asi-Albert Trees and ER graphs} \label{BA_section}

%{\color{red}Octavio: Poner intro y resultados anteriores} 
One of the motivations for the main inequality is to give better bounds for random graphs which have many leaves or, alternatively, they are sparse graphs. We will consider two regimes where this hypothesis hold: trees chosen with the preferential model of Barab\'asi-Albert \cite{Barabasi99} and ER graphs \cite{Erdos59} in the sparse regime. %In this section and the subsequent we will say that an event occurs with high probability if  the pa 

Let consider a sequence of graphs \((G_n)_{n \in \N} = ((V_n,E_n))_{n \in \N}\) where $G_n$ is a graph in $n$ nodes for every $n \in \N$. In the case of trees, we prefer the notation $(T_n)_{n \in \N}$ instead. We define, for every $n \in \N$,
\[
    N_k(n) := \{v \in V_n : d(v) = k\}
\]
to be the set of vertices of degree \(k\) in \(G_n\) and
\[
    N_{k,1}(n) := \{v \in L_n : w \in N(v),\, d(w) = k\}
\]
the set of leafs in \(G_n\) whose neighbor has degree \(k\).

\subsection{Barab\'asi-Albert Trees}

In that case of Barab\'asi-Albert trees, the energy is asymptotically smaller than the size of the tree with probability tending to 1 as the size tends to infinity, as proved in \cite{Arizmendi22}. In this section  we improve the bound obtained in \cite{Arizmendi22}, getting closer to the actual limit of the energy. 

In this section, we will consider a sequence of trees $(T_n)_{n \in \N}$ as before where $T_n$ is a tree typical chosen with Barab\'asi-Albert model. For this type of random trees, the asymptotic distribution of vertex of degree $k$ and edges joining a vertex of degree $k$ and a leaf can by understood explicitly. 

\begin{proposition}\label{BA-coefficients} 
\cite{Bollobas01, McDonald}
With high probability \(\abs{N_k(n)}/n \to n_k\) and \(\abs{N_{k,1}(n)}/n \to n_{k,1}\) when \(n \to \infty\), where
\[
    n_k = \frac{4}{k(k+1)(k+2)}, \qquad n_{k,1} = \frac{2(k+2)(k+3)-24}{k(k+1)(k+2)(k+3)}.
\]
\end{proposition}

Using this proposition and the bound obtained in Theorem \ref{TP}, we achieve to improve the bound for the asymptotic behavior of the energy for a tree following the model of Barab\'asi-Albert.

\begin{theorem}\label{Asym_BA}
With high probability, as $n$ tends to infinity, the next inequality follows
$$\limsup_{n\to\infty} \frac{\energy{T_n}}{n} \leq 0.96.$$

\end{theorem}

\begin{proof}
By the Theorem \ref{TP}, we get the next inequality:
\begin{align*}
    \energy{T_n} &\leq \sum_{v \in V'} \sqrt{3l(v) + d(v)}\\
    &\leq \sum_{k = 2}^n \sum_{v \in N_k(n)} \sqrt{3l(v) + k}.
\end{align*}

For each \(k \geq 2\), by using Cauchy-Schwartz inequality,
\begin{align*}
    \sum_{v \in N_k(n)} \sqrt{3l(v) + k} &\leq \sqrt{\abs{N_k(n)}\left(\sum_{v \in N_k(n)} 3l(v) + k\right)}\\
    &= \sqrt{\abs{N_k(n)}\left(3\abs{N_{k,1}(n)} + k\abs{N_k(n)}\right)}\\
    &= \abs{N_k(n)}\sqrt{3\frac{\abs{N_{k,1}(n)}}{\abs{N_k(n)}} + k}.
\end{align*}

Combining the previous calculations with Proposition \ref{BA-coefficients}, for any fixed \(m \geq 2\) and large \(n\), we get
\begin{align*}
    \frac{1}{n}\energy{T_n} &\leq \sum_{k = 2}^{n} \frac{\abs{N_k(n)}}{n} \sqrt{3\frac{\abs{N_{k,1}(n)}}{\abs{N_k(n)}} + k}\\
    &= \sum_{k = 2}^{m} n_k \sqrt{3\frac{n_{k,1}}{n_k} + k} + o(1) + \sum_{k = m+1}^n \frac{\abs{N_k(n)}}{n} \sqrt{3\frac{\abs{N_{k,1}(n)}}{\abs{N_k(n)}} + k}.
\end{align*}

Note that, for \(v \in N_k(n)\), we have at most \(d(v) = k\) neighbors that are leaf. Then, we get
\[
    \abs{N_{k,1}(n)} \leq k\abs{N_k(n)}, \qquad\mbox{for every } k \geq 2.
\]
Moreover, we have the next inequality (see \cite{Arizmendi22})
\[
    \frac{1}{n}\sum_{k = m+1}^n \abs{N_k(n)}\sqrt{k} \leq \frac{2}{m} + o(1).
\]

This way,
\begin{align*}
    \frac{1}{n}\energy{T_n} &\leq \sum_{k = 2}^{m} n_k \sqrt{3\frac{n_{k,1}}{n_k} + k} + o(1) + \frac{2}{n}\sum_{k = m+1}^n \abs{N_k(n)}\sqrt{k}\\
    &= \sum_{k = 2}^m n_k \sqrt{3\frac{n_{k,1}}{n_k} + k} + \frac{4}{m} + o(1),
\end{align*}
where
\[
    n_k = \frac{4}{k(k+1)(k+2)}, \qquad n_{k,1} = \frac{2(k+2)(k+3)-24}{k(k+1)(k+2)(k+3)}.
\]

Since \(m\) is fixed but arbitrary, we obtain the asymptotic bound
\begin{align*}
    \limsup_{n \to \infty} \frac{1}{n}\energy{T_n} &\leq \sum_{k = 2}^{\infty} n_k\sqrt{3\frac{n_{k,1}}{n_k} + k}\\
    &= \sum_{k = 2}^{\infty} \frac{4}{k(k+1)(k+2)} \sqrt{\frac{k(5k+21)-18}{2(k+3)}}.
\end{align*}

Notice that, for a fixed $m \geq 21$,
\begin{align*}
    \delta &:= \sum_{k = m}^{\infty} n_k \sqrt{3\frac{n_{k,1}}{n_k} + k}
    \leq \sum_{k = m}^{\infty} \frac{4\sqrt{3}}{k^{5/2}}\\
    &\leq \int_{m-1}^{\infty} \frac{4\sqrt{3}}{x^{5/2}}dx = \frac{8\sqrt{3}}{3(m-1)^{3/2}}.
\end{align*}

Taking $m = 10^5 + 1$, we have $\delta = \frac{8\sqrt{3}}{3\cdot 10^{15/2}} \approx 1.46059 \cdot 10^{-7}$. Therefore,
\[
    \limsup_{n \to \infty} \frac{1}{n}\energy{G} \leq \sum_{k = 2}^{m} n_k \sqrt{3\frac{n_{k,1}}{n_k} + k} + \delta \approx 0.95999.
\]
    
\end{proof}

Notice from the proof that we used the following observation which we state as a corollary, for further reference.

\begin{corollary}\label{cor-BA-model}
For a graph $G = (E,V)$ it is satisfied that
\[
    \energy{G} \leq 2e_{11} + \sum_{k = 1}^n n_k\sqrt{3\frac{n_{k,1}}{n_k} + k},
\]
where $n_k$ is the number of nodes in $G$ with degree $k$ and $n_{k,1}$ is the number of leaves with only neighbor having degree $k$.

\end{corollary}

\begin{figure}[H]
    \centering
    \includegraphics[width = 0.7\linewidth]{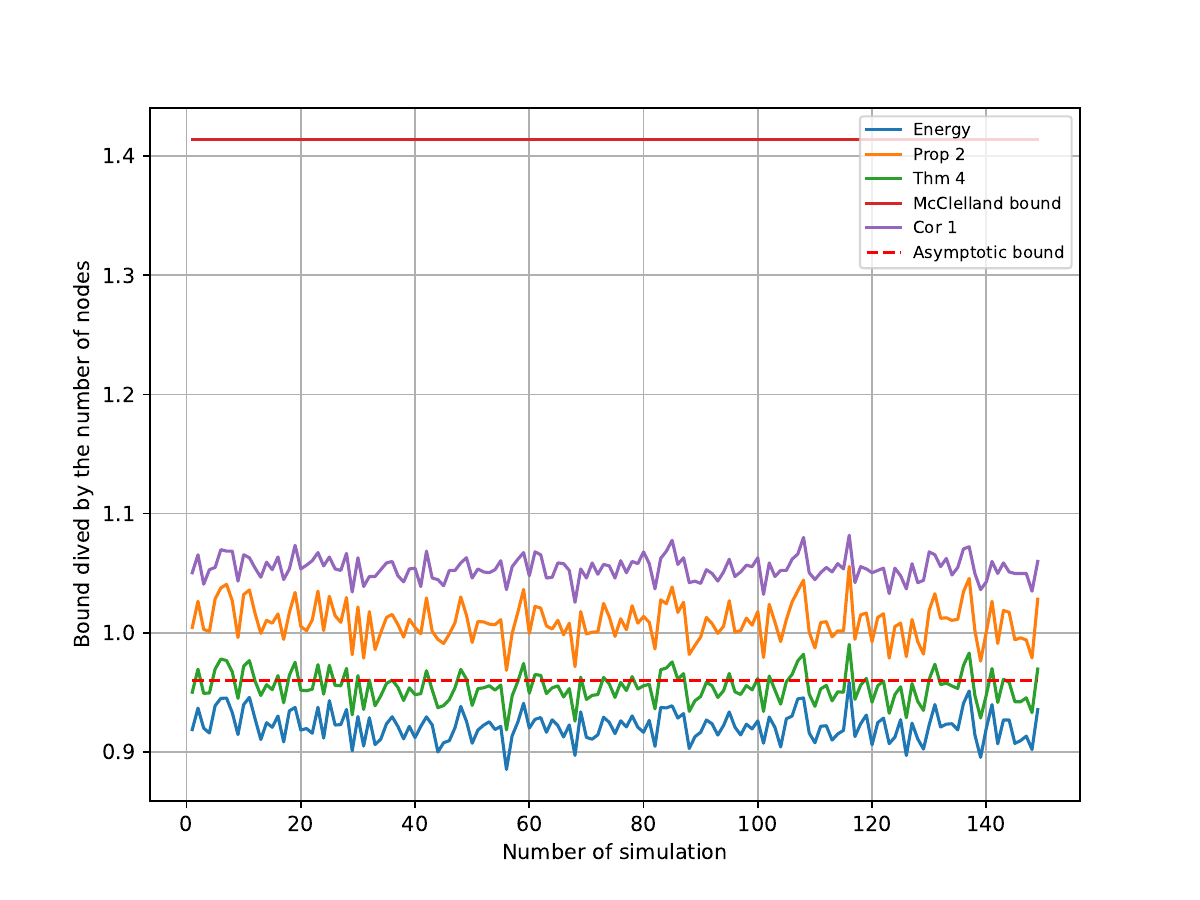}
    \caption{Energy / Size for 200 random trees of size \(n = 2000\) following the Barab\'asi-Albert model compared with bounds from Proposition \ref{eq:AD}, Theorem \ref{TP}, McClleland's bound, Corollary \ref{global_bound} and the asymptotic bound from Theorem \ref{Asym_BA}.}
\end{figure}

\subsection{Energy for Erd\"os-R\'enyi Graphs}
\label{ER}

Similarly to the case for Barab\'asi-Albert trees, above, we will derive an asymptotic bound for Erd\"os-R\'enyi graphs in the sparse regime. Let us mention that the case $p_n\to p>0$ follows from Wigner's Semicircle law as observed by Nikiforov in \cite{Nikiforov}.

In this section, we will consider a sequence $(G_n)_{n \in \N}$ as before where $G_n$ is a graph typical chosen with Erd\"os-R\'enyi model \(G(n,p_n)\), such that \(np_n \to \lambda\) when \(n \to \infty\) for some constant \(\lambda > 0\).

The following proposition tells what is the asymptotic behavior of \(\abs{N_k(n)}\) and  \(\abs{N_{k,1}(n)}\) when \(n\) tends to \(\infty\). The result is probably known to the expert and our proof is standard, but we include it for the convenience of the reader.
\begin{proposition}\label{ER-deg}
    With high probability \(\abs{N_k(n)}/n \to n_k\) and \(\abs{N_{k,1}(n)}/n \to n_{k,1}\) when \(n \to \infty\), where
    \[
    n_k = \frac{\lambda^ke^{-\lambda}}{k!}, \qquad n_{k,1} = \frac{\lambda^ke^{-2\lambda}}{(k-1)!}.
    \]
\end{proposition}

\begin{proof}
    Let \(n,k \in \N\) be fixed, with \(k < n\), and \(v \in V_n\). Notice that
    \[
        \Pp{v \in N_k(n)} = \binom{n-1}{k}p^k(1-p)^{n-1-k}
    \]
    and
    \[
        \lim_{n \to \infty} \Pp{v \in N_k(n)} = \lim_{n \to \infty} \frac{(n)_k}{n^k}\frac{(np)^k}{k!} \left(1 - \frac{np}{n}\right)^{n-1-k} = n_k.
    \]
    
    However,
    \[
        \abs{N_k(n)} = \sum_{v \in V_n} \mathbbm{1}_{v \in N_k(n)}.
    \]
    Then,
    \[
        \E{\frac{\abs{N_k(n)}}{n}} = \frac{1}{n}\sum_{v \in V_n} \E{\mathbbm{1}_{v \in N_k(n)}} = \Pp{v_0 \in N_k(n)},
    \]
    where \(v_0 \in V_n\) is a fixed vertex. In consequence,
    \[
        \lim_{n \to \infty} \E{\frac{\abs{N_k(n)}}{n}} = \lim_{n \to \infty} \Pp{v \in N_k(n)} = n_k.
    \]

    Let \(v,w \in V_n\) be distinct vertices. Considering the possibilities: \(v \sim w\) or \(v \not\sim w\); and the number of shared neighbors \(j\) between \(v\) and \(w\), we obtained
    \begin{align*}
        \Pp{v,w \in N_k(n)} &= \sum_{j = 0}^k \binom{n-2}{j}\binom{n-2-j}{k-j}\binom{n-2-k}{k-j}p^{2k}(1-p)^{2n-2k-3}\\
        &\quad + \sum_{j = 0}^{k-1} \binom{n-2}{j}\binom{n-2-j}{k-j-1}\binom{n-1-k}{k-j-1}p^{2k-1}(1-p)^{2n-2k-2)}.
    \end{align*}
    Then, using a similar limit to the previously used, we get
    \[
        \lim_{n \to \infty} \Pp{v,w \in N_k(n)} = n_k^2.
    \]
    
    However,
    \begin{align*}
        \Var{\frac{\abs{N_k(n)}}{n}} &= \frac{1}{n^2}\left[\sum_{v \in V_n}\E{\mathbbm{1}_{v \in N_k(n)}} + \sum_{\substack{v,w \in V_n\\
        v \neq w}} \E{\mathbbm{1}_{v, w \in N_k(n)}}\right] - \E{\frac{\abs{N_k(n)}}{n}}^2\\
        &= \frac{1}{n}\Pp{v_0 \in N_k(n)} + \frac{n-1}{n}\Pp{v_0,w_0 \in N_k(n)} - \Pp{v \in N_k(n)}^2,
    \end{align*}
    where \(v_0,w_0 \in V\) are fixed. In consequence,
    \[
        \lim_{n \to \infty} \Var{\frac{\abs{N_k(n)}}{n}} = 0.
    \]
    Therefore, \(\abs{N_k(n)}/n\) tends to \(n_k\) almost surely while \(n\) tends to \(\infty\).

    Let again \(n,k \in \N\) be fixed, with \(k < n\), and \(v \in V_n\). Note that
    \[
        \Pp{v \in N_{k,1}(n)} = (n-1)\binom{n-2}{k-1}p^k(1-p)^{2n-k-3}
    \]
    and
    \[
        \lim_{n \to \infty} \Pp{v \in N_{k,1}(n)} = n_{k,1}.
    \]

    Nevertheless,
    \[
        \abs{N_{k,1}(n)} = \sum_{v \in V_n} \mathbbm{1}_{v \in N_{k,1}(n)}.
    \]
    It follows that
    \[
        \E{\frac{\abs{N_{k,1}(n)}}{n}} = \frac{1}{n}\sum_{v \in V_n} \E{\mathbbm{1}_{v \in N_{k,1}(n)}} = \Pp{v \in N_{k,1}(n)}.
    \]
    As a result,
    \[
        \lim_{n \to \infty} \E{\frac{\abs{N_{k,1}(n)}}{n}} = \Pp{v \in N_{k,1}(n)} = n_{k,1}.
    \]

    Let \(v,w \in V_n\) be distinct. If \(v, w \in N_{k,1}(n)\), there exist \(v',w' \in V_n\) such that \(v \sim v'\), \(w \sim w\) and \(d(v') = k = d(w')\). We have several options:
    \begin{itemize}
        \item \(v' = w'\);
        \item \(v' \neq w'\) and \(v' \not\sim w'\);
        \item \(v' \neq w'\) and \(v' \sim w'\).
    \end{itemize}
    Considering the number of shared neighbors \(j\) and \(w'\), we get
    \begin{align*}
        \Pp{v,w \in N_{k,1}(n)} &= (n-2)\binom{n-3}{k-2}p^k(1-p)^{3n-k-4}\\
        &\quad + \sum_{j = 0}^{k-1} (n-2)(n-3)\binom{n-4}{j}\binom{n-j-4}{k-j-1}\binom{n-k-3}{k-j-1}p^{2k}(1-p)^{4n-2k-10}\\
        &\quad + \sum_{j = 0}^{k-2} (n-2)(n-3)\binom{n-4}{j}\binom{n-j-4}{k-j-2}\binom{n-k-2}{k-j-2}p^{2k-1}(1-p)^{4n-2k-9}.
    \end{align*}
    Using a similar limit to the previously used, we obtained
    \[
        \lim_{n \to \infty} \Pp{v,w \in N_{k,1}(n)} = n_{k,1}^2.
    \]

    However,
    \begin{align*}
        \Var{\frac{\abs{N_{k,1}(n)}}{n}} &= \frac{1}{n^2}\left[\sum_{v \in V}\E{\mathbbm{1}_{v \in N_{k,1}(n)}} + \sum_{\substack{v,w \in V\\
        v \neq w}} \E{\mathbbm{1}_{v, w \in N_{k,1}(n)}}\right] - \E{\frac{\abs{N_{k,1}(n)}}{n}}^2\\
        &= \frac{1}{n}\Pp{v_0 \in N_{k,1}(n)} + \frac{n-1}{n}\Pp{v_0,w_0 \in N_{k,1}(n)} - \Pp{v \in N_{k,1}(n)}^2,
    \end{align*}
    where \(v_0,w_0 \in V_n\) are fixed. In consequence,
    \[
        \lim_{n \to \infty} \Var{\frac{\abs{N_{k,1}(n)}}{n}} = 0.
    \]
    Therefore, \(\abs{N_{k,1}(n)}/n\) tends to \(n_{k,1}\) almost surely while \(n\) tends to \(\infty\).
    
\end{proof}

%{\color{red} Tal vez esta prueba está de más? Sugiero esto. We will also use the next simple result, which follows by a direct application of the law of large numbers.}

\begin{proposition}\label{ER-edges}
%With high probability, 
\(\abs{E_n}/n \to \lambda/2\) as \(n \to \infty\), almost surely.
\end{proposition}

%\begin{proof}
%Note that
%\[
%    \frac{1}{n}\abs{E_n} = \frac{1}{n}\sum_{\substack{v,w \in V_n\\ v \neq w}} \mathbbm{1}_{\{v,w\} \in E_n}.
%\]
%Since the probability of an edge \(\{v,w\}\) is \(p_n\), for every \(v,w \in V_n\) distinct, and every edge appears or not independent from the other, we have
%\[
%    \E{\frac{1}{n}\abs{E_n}} = \binom{n}{2}\frac{p_n}{n}
%\]
%and
%\[
%    \Var{\frac{1}{n}\abs{E_n}} = \binom{n}{2}\frac{p_n(1-p_n)}{n^2}.
%\]
%Then
%\[
%    \lim_{n \to \infty} \E{\frac{1}{n}\abs{E_n}} = \frac{\lambda}{2}, \qquad \lim_{n \to \infty} \Var{\frac{1}{n}\abs{E_n}} = 0.
%\]
%Therefore, with high probability \(\abs{E_n}/n \to \lambda/2\) while \(n \to \infty\).

%\end{proof}
Now we proceed to bound the energy of the graph $G_n$, asymptotically.
\begin{theorem}\label{ER-bound1}
    With high probability,
  \begin{eqnarray}
    \limsup_{n \to \infty} \frac{\energy{G_n}}{n} &\leq& 2\lambda e^{-2\lambda} + e^{-\lambda}\sqrt{3e^{-\lambda} + 1}\sum_{k = 2}^{\infty} \frac{\lambda^k\sqrt{k}}{k!} 
\end{eqnarray}

\end{theorem}

\begin{proof}
Let
\[
    n_k = \frac{\lambda^ke^{-\lambda}}{k!}, \qquad n_{k,1} = \frac{\lambda^ke^{-2\lambda}}{(k-1)!}.
\]

Notice that the total number of vertices may be calculated by
\begin{equation}\label{eq-sum-nk}
    \sum_{k = 1}^n \abs{N_k(n)} = n,
\end{equation}
while the sum of degrees may be calculated by
\begin{equation}\label{eq-sum-deg}
    \sum_{k = 1}^n k \abs{N_k(n)} = 2\abs{E_n} = n\lambda + o(n),
\end{equation}
using the Proposition \ref{ER-edges}.

In the other hand, for any fixed \(m \in \N\),
\[
    \sum_{k = 1}^m k \,n_k = \lambda e^{-\lambda}\sum_{k = 0}^{m-1} \frac{\lambda^k}{k!} \geq \lambda - \frac{\lambda^{m+1}}{m!} + o(1),
\]
where the second inequality is result of Taylor's Theorem used in \(e^{\lambda}\),
\[
    e^{\lambda} = \sum_{k = 0}^{m-1} \frac{\lambda^k}{k!} + \frac{\lambda^m e^{\lambda'}}{m!} + o(1), \qquad \mbox{for some } \lambda' \in (0,\lambda).
\]
Comparing with the equation \eqref{eq-sum-deg}, we infer that
\[
    \frac{1}{n}\sum_{k = m+1}^n k\abs{N_k(n)} \leq \frac{\lambda^{m+1}}{m!} + o(1).
\]
From this equation, together with equation \eqref{eq-sum-nk}, by Cauchy-Schwarz we obtain that
\[
    \frac{1}{n}\sum_{k = m+1}^n \abs{N_k(n)}\sqrt{k} \leq \sqrt{\frac{\lambda^{m+1}}{m!}} + o(1).
\]

By the Corollary \ref{cor-BA-model}, we get that
\[
    \frac{\energy{G_n}}{n} \leq 2\frac{\abs{N_{1,1}(n)}}{n} + \sum_{k = 2}^n \frac{\abs{N_k(n)}}{n} \sqrt{3 \frac{\abs{N_{k,1}(n)}}{\abs{N_k(n)}} + k}
\]
Combining this equation with the previous calculations, for any fixed \(m \geq 2\) and large \(n\), we get
\begin{align*}
    \frac{1}{n}\energy{G_n} &\leq 2n_{1,1} + \sum_{k = 2}^m n_k\sqrt{3\frac{n_{k,1}}{n_k} + k} + o(1) + \frac{2}{n}\sum_{k = m+1}^n \abs{N_k(n)}\sqrt{k}\\
    &\leq 2n_{1,1} + \sum_{k = 2}^m n_k\sqrt{3\frac{n_{k,1}}{n_k} + k} + 2\sqrt{\frac{\lambda^{m+1}}{m!}} + o(1)
\end{align*}

Since \(m\) is fixed but arbitrary, we obtain the asymptotic bound
\begin{align*}
    \limsup_{n \to \infty} \frac{1}{n}\energy{G_n} &\leq 2n_{1,1} + \sum_{k = 2}^{\infty} n_k\sqrt{3\frac{n_{k,1}}{n_k} + k}\\
    &= 2\lambda e^{-2\lambda} + e^{-\lambda}\sqrt{3e^{-\lambda} + 1}\sum_{k = 2}^{\infty} \frac{\lambda^k\sqrt{k}}{k!}.
\end{align*}
    
\end{proof}

\begin{figure}[!ht]
    \centering

    \begin{subfigure}[t]{0.49\textwidth}
        \centering
        \includegraphics[width = 1\linewidth]{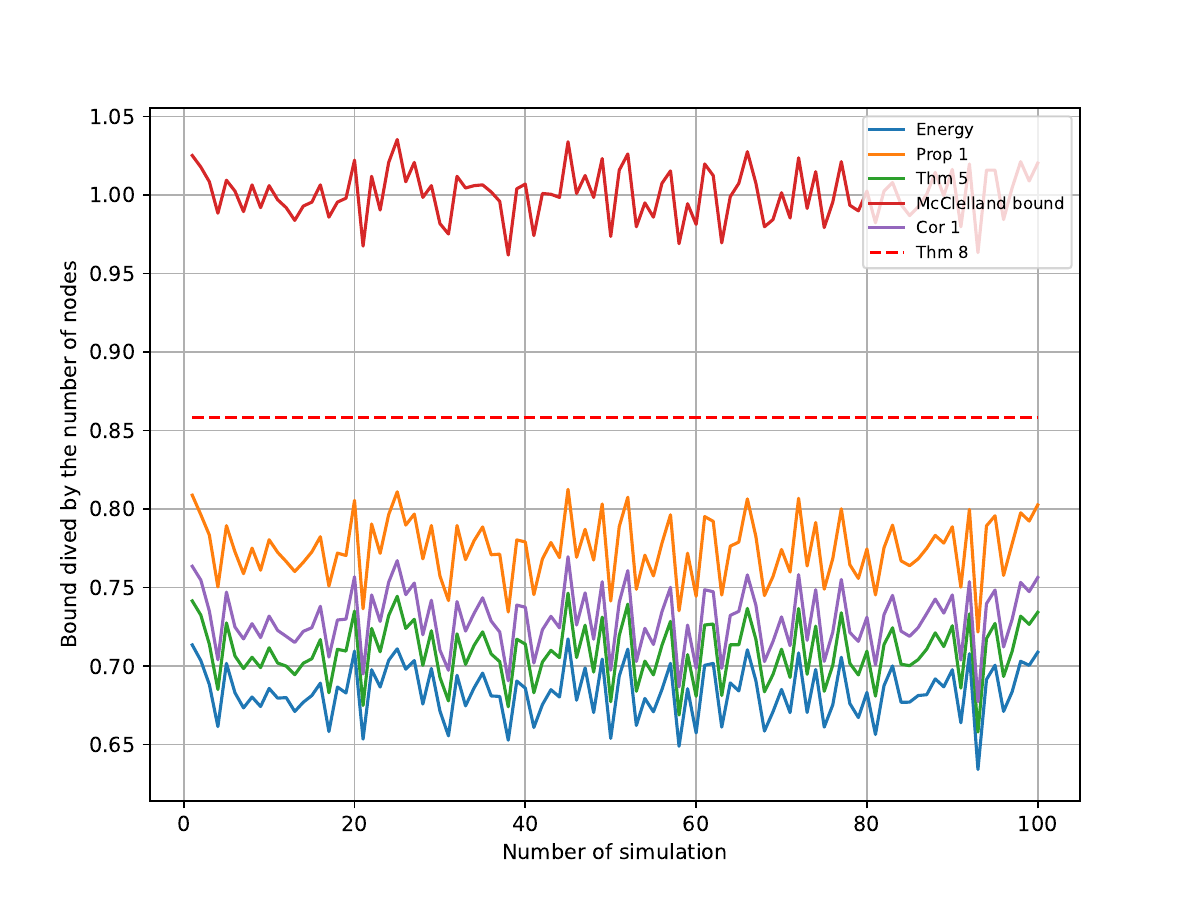}
        \caption{$\lambda = 1$}
    \end{subfigure}%
    ~
    \begin{subfigure}[t]{0.49\textwidth}
        \centering
        \includegraphics[width = 1\textwidth]{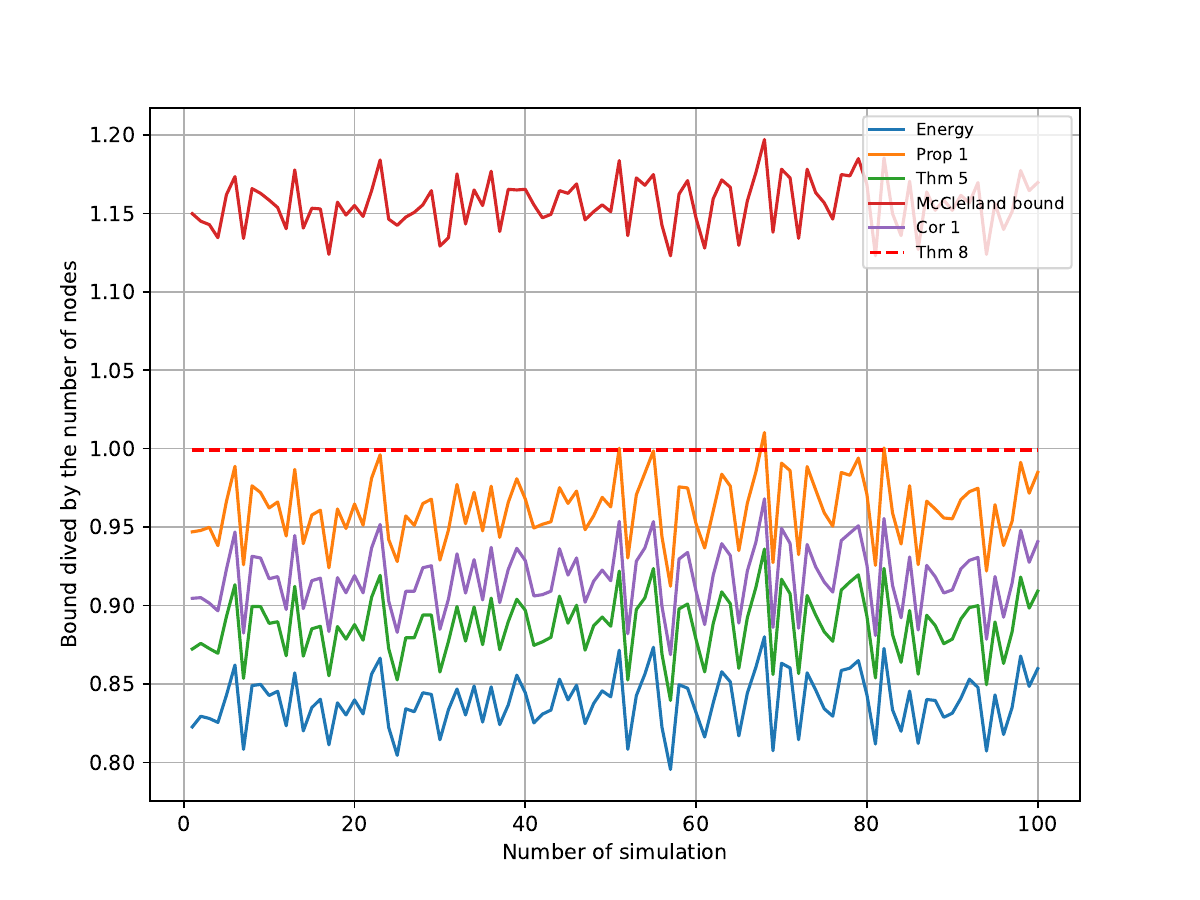}
        \caption{$\lambda = 4/3$}
    \end{subfigure}
    
    \begin{subfigure}[t]{0.49\textwidth}
        \centering
        \includegraphics[width = 1\textwidth]{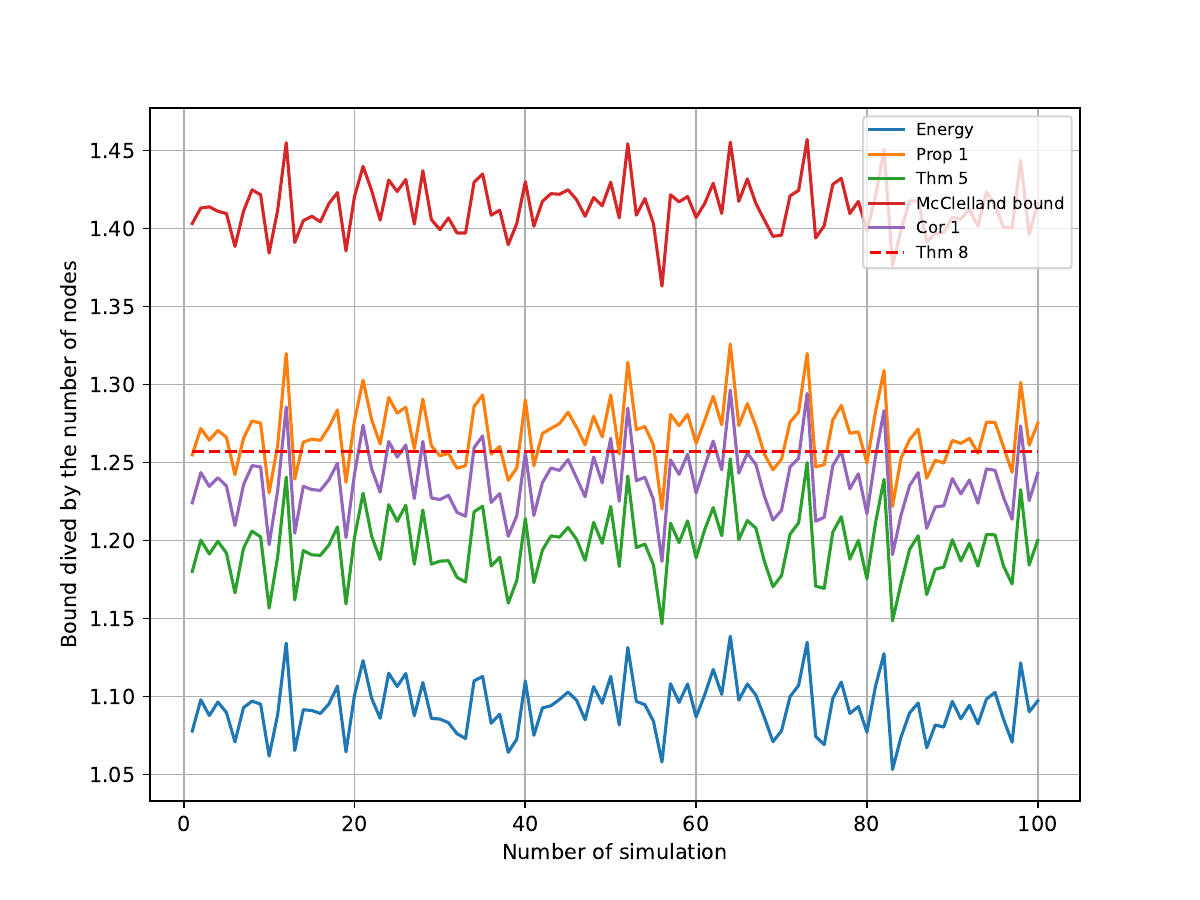}
        \caption{$\lambda = 2$}
    \end{subfigure}
    
    \caption{Energy/size for 200 random graphs of size \(n = 2000\) following the Erd\"os-Renyi model compared with the bounds from Proposition \ref{eq:AJ}, Theorem \ref{TPG}, McCleland's bound, Corollary \ref{global_bound} and the asymptotic bound from Theorem \ref{ER-bound1} for different values of $\lambda$.}
\end{figure}

As a consequence of the above result we can ensure that as $n$ aproaches infinity an Erd\"os-Renyi graph is hypoenergetic when taking $p_n \leq \frac{4}{3n}$, as the following result states.

\begin{theorem}\label{ER-hipoenergetic}
For \(\lambda \in (0,4/3]\), it is satisfied that
\[
    \limsup_{n \to \infty} \frac{\energy{G_n}}{n} < 1.
\]
\end{theorem}

\begin{proof}
Let $f : (0,\infty) \to \R$ given by
\[
    f(\lambda) = 2\lambda e^{-2\lambda} + e^{-\lambda}\sqrt{3e^{-\lambda} + 1}\sum_{k = 2}^{\infty} \frac{\lambda^k\sqrt{k}}{k!}.
\]
We note first that $f$ is increasing for $\lambda \in [0,4/3]$. Indeed,
\[
    f'(\lambda) = g(\lambda) + \sum_{k = 5}^{\infty} h_k(\lambda),
\]
where, for every \(k \geq 2\),
\[
    g(\lambda) := (2 - 4\lambda)e^{-2\lambda} + \sum_{k = 2}^4 h_k(\lambda), \qquad h_k(\lambda) := \frac{\sqrt{k}e^{-2\lambda}\left(2e^{\lambda}(k - \lambda) + 6k - 9\lambda\right)}{2k!\sqrt{3e^{-\lambda} + 1}}.
\]
It is clear that \(h_k > 0\) whenever \(2e^{\lambda}(k - \lambda) + 6k - 9\lambda > 0\), in particular, when $k\geq 5$ we can take  \(\lambda \in (0,3]\). In addition, after some calculations one can verify that \(g(\lambda) > 0\) for \(\lambda \in (0,4/3]\). In consequence \(f\) is increasing in \((0,4/3]\).

Thus it is enough to check if $f(4/3)\leq 1$.  To bound $f(4/3)$ we can consider the partial sum \[
    f_n(\lambda) = 2\lambda e^{-2\lambda} + e^{-\lambda}\sqrt{3e^{-\lambda} + 1}\sum_{k = 2}^{n} \frac{\lambda^k\sqrt{k}}{k!}.
\]
and verify that $|f(\lambda)-f_n(\lambda)|\leq\sqrt{3e^{-\lambda} + 1}\lambda^n/n!$, which for $\lambda=4/3$ and $n = 13$ yields and error of $\delta \approx 9.04577\cdot 10^{-9}$, which is a small as desired as $n\to \infty$.
In conclusion, we obtain that \(f(4/3) \approx 0.99911\) and therefore, for every \(\lambda \in (0,4/3]\) it is satisfied that
\[
    \limsup_{n \to \infty} \frac{\energy{G_n}}{n} \leq f(\lambda) < 1.
\]

\end{proof}

We finally present some general bounds, which are particularly accurate when $\lambda$ is small.

\begin{proposition}\label{ER-bound2}
 Let 
\begin{equation*}
   f(\lambda) := 2\lambda e^{-2\lambda} + e^{-\lambda}\sqrt{3e^{-\lambda} + 1}\sum_{k = 2}^{\infty} \frac{\lambda^k\sqrt{k}}{k!}
\end{equation*}

Then
\[
    f(\lambda) \leq 2\lambda e^{-2\lambda} + \frac{1}{\sqrt{8}}\sqrt{3e^{-\lambda} + 1}\left((\lambda + 2) - e^{-\lambda}(3\lambda + 2)\right), \qquad \lambda \in (0,\infty)
\]
and
\[
    f(\lambda) \leq \lambda, \qquad \lambda \in (4/3,\infty).
\]
\end{proposition}

\begin{proof}
Using Taylor expansion around \(2\) for $\sqrt{\cdot}$ function, we notice that
\[
    \sqrt{x} \leq \frac{1}{\sqrt{8}} x + \frac{1}{\sqrt{2}}, \qquad \mbox{for } x \geq 2.
\]
Then,
\begin{align*}
    \sum_{k = 2}^{\infty} \frac{\lambda^k\sqrt{k}}{k!} &\leq \frac{1}{\sqrt{8}}\left(\lambda\sum_{k = 1}^{\infty} \frac{\lambda^k}{k!} + 2\sum_{k = 2}^{\infty} \frac{\lambda^k}{k!}\right)\\
    &= \frac{\lambda(e^{\lambda} - 1) + 2(e^{\lambda} - \lambda - 1)}{\sqrt{8}}\\
    &= \frac{(\lambda + 2)e^{\lambda} - (3\lambda + 2)}{\sqrt{8}}.
\end{align*}

Therefore, for \(\lambda \in (0,\infty)\),
\[
    f(\lambda) \leq 2\lambda e^{-2\lambda} + \frac{1}{\sqrt{8}}\sqrt{3e^{-\lambda} + 1}\left(\lambda + 2 - e^{-\lambda}(3\lambda + 2)\right).
\]

Let \(g : \R_+ \to \R\) be the function given by
\[
    g(\lambda) = 2\lambda e^{-2\lambda} + \frac{1}{\sqrt{8}}\sqrt{3e^{-\lambda} + 1}\left(\lambda + 2 - e^{-\lambda}(3\lambda + 2)\right) - \lambda.
\]
We obtain that
\[
    g'(\lambda) = \frac{2\sqrt{2}e^{\lambda} + \sqrt{2}(3\lambda - 2) + 27\sqrt{2}\lambda e^{-\lambda} - 16\sqrt{3e^{-\lambda} + 1}(\lambda - 1)}{8e^{\lambda}\sqrt{3e^{-\lambda} + 1}} - 1.
\]
By writing $g'(\lambda)=g_1(\lambda)+ g_2(\lambda)+g_3(\lambda)-1$
where, for $\lambda\geq 4/3,$ $$g_1(\lambda)=\frac{2\sqrt{2}e^{\lambda}}{{8e^{\lambda}\sqrt{3e^{-\lambda} + 1}}}\leq \frac{1}{2\sqrt{2}},\quad  g_2(\lambda)=\frac{27\sqrt{2}\lambda e^{-\lambda}}{8e^{\lambda}\sqrt{3e^{-\lambda} + 1}}\leq \frac{1}{2}, $$
and
$$g_3(\lambda)=\frac{\sqrt{2}(3\lambda - 2) - 16\sqrt{3e^{-\lambda} + 1}(\lambda - 1)}{8e^{\lambda}\sqrt{3e^{-\lambda} + 1}} \leq \frac{\sqrt{2}(3\lambda - 2) - 16(\lambda - 1)}{8e^{\lambda}\sqrt{3e^{-\lambda} + 1}}\leq 0$$

One sees that $g'(\lambda)\leq0$ for \(\lambda>4/3\). 
%Now, it is clear that for \(\lambda\) sufficiently big, we have the following inequality:
%\[   \sqrt{2}(3\lambda - 2) + 27 \sqrt{2} \lambda e^{-\lambda} \leq 16 (\lambda - 1) + (8 - 2\sqrt{2})e^{\lambda}.
%\]
%In particular, for \(\lambda = \)

\end{proof}

\begin{figure}[!ht]
    \centering
    \includegraphics[width = 0.7\linewidth]{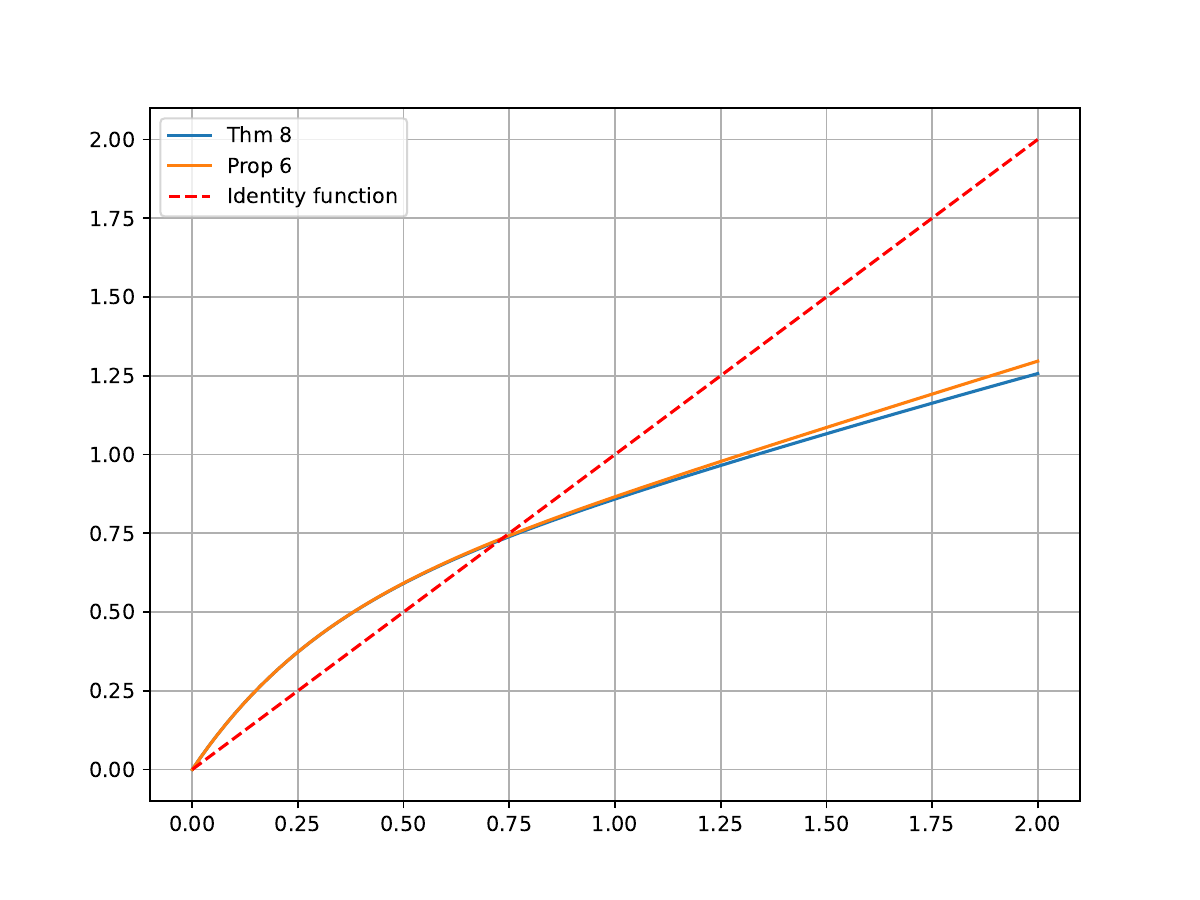}
    \caption{Bounds for the energy/size for a typical random graph chosen by Erd\"os-Renyi model when the number of nodes tends to infinity given by Theorem \ref{ER-bound1} and Proposition \ref{ER-bound2}.}
\end{figure}

\end{document}